\documentclass[12pt,reqno]{amsart}

\newcommand\version{December 8, 2017}

\usepackage{amsmath,amsfonts,amsthm,amssymb,amsxtra}
\usepackage{bbm} 

\newtheorem{theorem}{Theorem}[section]
\newtheorem{proposition}[theorem]{Proposition}
\newtheorem{lemma}[theorem]{Lemma}
\newtheorem{corollary}[theorem]{Corollary}

\theoremstyle{definition}

\theoremstyle{remark}

\numberwithin{equation}{section}

\newcommand{\1}{\mathbbm{1}}

\newcommand{\C}{\mathbb{C}}

\renewcommand{\epsilon}{\varepsilon}

\newcommand{\N}{\mathbb{N}}

\renewcommand{\phi}{\varphi}
\newcommand{\R}{\mathbb{R}}

\DeclareMathOperator{\supp}{supp}
\DeclareMathOperator{\sgn}{sgn}

\begin{document}

\title[Two-dimensional Schr\"odinger operators on domains --- \version]{Bound on the number of negative eigenvalues of two-dimensional Schr\"odinger operators on domains}

\author{Rupert L. Frank}
\address[R. Frank]{Mathematisches Institut, Ludwig-Maximilans Universit\"at M\"unchen, Theresienstr. 39, 80333 M\"unchen, Germany, and Department of Mathematics, California Institute of Technology, Pasadena, CA 91125, USA}
\email{r.frank@lmu.de}

\author{Ari Laptev}
\address[A. Laptev]{Department of Mathematics, Imperial College London, 180 Queen's Gate, London SW7 2AZ, UK, and Institut Mittag--Leffler, Aurav\"agen 17, 182 60 Djursholm, Sweden}
\email{laptev@mittag-leffler.se}

\begin{abstract}
A fundamental result of Solomyak says that the number of negative eigenvalues of a Schr\"odinger operator on a two-dimensional domain is bounded from above by a constant times a certain Orlicz norm of the potential. Here we show that in the case of Dirichlet boundary conditions the constant in this bound can be chosen independently of the domain.
\end{abstract}

\dedicatory{Dedicated to the memory of M. Z. Solomyak}

\renewcommand{\thefootnote}{${}$} \footnotetext{\copyright\, 2017 by
  the authors. This paper may be reproduced, in its entirety, for
  non-commercial purposes.}

\maketitle

\section{Introduction and main result}

In this paper we are interested in an upper bound on the number of negative eigenvalues of a Schr\"odinger operator in two space dimensions. In order to state our main result, we introduce the functions 
$$
\mathcal A(t):=e^{|t|}-1-|t|
\qquad\text{and}\qquad
\mathcal B(t):= (1+|t|)\ln(1+|t|) -|t| \,.
$$
These functions are convex and Legendre transforms of each other. For a measurable set $\Omega\subset\R^2$ of positive and finite measure we denote by $L^{\mathcal B}(\Omega)$ the class of (almost everywhere equal) measurable functions on $\Omega$ for which
$$
\|g \|_{\mathcal B,\Omega} := \sup \left\{ \left| \int_\Omega fg \,dx \right| :\ f:\Omega\to\C \ \text{measurable such that}\ \int_\Omega \mathcal A(f) \,dx \leq |\Omega| \right\}
$$
is finite. Also $-\Delta_\Omega^D$ denotes the Dirichlet Laplacian on $\Omega$ and $N(H)$ denotes the number of negative eigenvalues, counting multiplicities, of a self-adjoint lower semi-bounded operator $H$.

Our main result is

\begin{theorem}\label{ch3:solomyak}
There is a constant $C$ such that for any open set $\Omega\subset\R^2$ of finite measure and any real $V\in L^\mathcal B(\Omega)$,
\begin{equation}
\label{eq:solomyak}
N(-\Delta_\Omega^D+V) \leq C \|V_-\|_{\mathcal B,\Omega} \,.
\end{equation}
\end{theorem}

Here we use the notation $V(x)_-=\max\{0,-V(x)\}$ for the negative part. 

The main point of this paper is that the constant $C$ in \eqref{eq:solomyak} can be chosen independently of $\Omega$. In a fundamental paper \cite{So} Solomyak had shown that
\begin{equation}
\label{eq:solomyakn}
N(-\Delta_\Omega^N+V) \leq 1+C_\Omega \|V_-\|_{\mathcal B,\Omega}
\end{equation}
for the Neumann Laplacian $-\Delta_\Omega^N$ and for bounded, connected open sets $\Omega\subset\R^2$ with the extension property. (This is essentially Theorem 4' in \cite{So}.) The constant $C_\Omega$ in \eqref{eq:solomyakn} depends on $\Omega$. From this bound and ideas in its proof it is easy to also obtain \eqref{eq:solomyak} for bounded open sets $\Omega\subset\R^2$, but again with a constant that depends on $\Omega$. The proof that this constant can, in fact, be chosen independently of $\Omega$ needs additional ingredients, which we provide in this paper. We also mention that conversely, it is easy to deduce \eqref{eq:solomyakn} from \eqref{eq:solomyak}, but then a dependence of $\Omega$ enters through the use of an extension operator.

A standard application of Theorem \ref{ch3:solomyak}, even with a domain dependent constant, yields the validity of Weyl asymptotics in the strong coupling limit for rough potentials. It is crucial here that the right side of \eqref{eq:solomyak} is homogeneous of degree one with respect to $V$. This is a distinguishing feature of \eqref{eq:solomyak} and \eqref{eq:solomyakn} compared to other bounds discussed below.

\begin{corollary}\label{weyl}
If $\Omega\subset\R^2$ is an open set of finite measure and $V\in L^\mathcal B(\Omega)$ is real, then
$$
\lim_{\alpha\to\infty} \alpha^{-1} N(-\Delta_\Omega^D+\alpha V) = \frac{1}{4\pi} \int_\Omega V(x)_-\,dx \,.
$$
\end{corollary}

The basic strategy in our proof of Theorem \ref{ch3:solomyak} is similar to that of Solomyak's proof of \eqref{eq:solomyakn} which, in turn, follows the strategy of Rozenblum's proof of the CLR bound \cite{Ro1,Ro2}. The set $\Omega$ is covered by cubes which are chosen in such a way that the Schr\"odinger operator on each cube with Neumann boundary conditions has at most one negative eigenvalue. Therefore the number of negative eigenvalues is bounded by the number of cubes and it remains to bound the latter number. In contrast to earlier works Rozenblum used overlapping cubes and the selection of the cubes proceeds by the Besicovich covering lemma. In order to guarantee that the Schr\"odinger operator on each cube has at most one negative eigenvalue, the cubes are chosen such that the norm of $V$ on each cube is equal to a given small constant. This construction is considerably more difficult in the two-dimensional case (where one deals with an Orlicz norm) than in the higher-dimensional case (where one deals with the $L^{d/2}$-norm). In fact, the choice of the (non-standard) norm $\|\cdot\|_{\mathcal B,\Omega}$, which has suitable superadditivity properties, was one of the key insights in Solomyak's work.

The new difficulty that we have to face in this work is that intersections of cubes with $\Omega$ can have essentially an arbitrary shape and that our constants have to be in a certain sense uniform with respect to this shape. In contrast, Solomyak needed only to consider intersections of cubes with each other, that is, rectangles, and the necessary uniformity follows in a rather straightforward manner. In our case we obtain the uniformity by carefully reviewing the proof of the Trudinger inequality. 

\medskip

We end this introduction by placing our result in the context of eigenvalue bounds for Schr\"odinger operators. As is well-known, the two-dimensional case is a borderline case and is still not as well understood as the case of three and higher dimensions and the one-dimensional case. Recently, there have been several results on the two-dimensional case \cite{KhMaWu,MoVa,Sh,GrNa,LaSo,LaSo2}. Solomyak's pioneering paper \cite{So} had a profound influence on these developments and we would like to dedicate our results here, which are also a variation on the theme in \cite{So}, to his memory.

In order to understand the particularity of the two-dimensional case, we recall that  in dimensions $d\geq 3$ the number $N(-\Delta+V)$ of negative eigenvalues of the Schr\"odinger operator $-\Delta +V$ in $L^2(\R^d)$ is bounded by the Cwikel--Lieb--Rozenblum inequality as
\begin{equation}
\label{eq:clr}
N(-\Delta+V) \leq C_d \int_{\R^d} V(x)_-^{d/2}\,dx \,,
\end{equation}
where $C_d$ is a constant depending only on $d$. This inequality should be compared with the Weyl asymptotics
\begin{equation}
\label{eq:weyl}
\lim_{\alpha\to\infty} \alpha^{-d/2} N(-\Delta +\alpha V) = \frac{|\{ \xi\in\R^d:\ |\xi|\leq 1\}|}{(2\pi)^d} \int_{\R^d} V(x)_-^{d/2}\,dx \,.
\end{equation}
These asymptotics are initially proved for continuous, compactly supported $V$ and then extended, using \eqref{eq:clr}, to any $V\in L^{d/2}(\R^d)$. Moreover, it is easy to deduce from \eqref{eq:clr} that if for $V\leq 0$ one has $\limsup_{\alpha\to\infty} \alpha^{-d/2} N(-\Delta+\alpha V)<\infty$, then $V\in L^{d/2}(\R^d)$.

The analogue of \eqref{eq:clr} is not true in dimension two. Indeed, in two dimensions for any potential $V\not\equiv 0$ with $\int_{\R^2} V(x)\,dx\leq 0$ the operator $-\Delta+V$ has a negative eigenvalue. However, not even a modified bound with the right side in \eqref{eq:clr} replaced by $C(1+ \int_{\R^2} V(x)_-\,dx)$ can hold, since a more subtle failure of \eqref{eq:clr} was discovered in \cite{BiLa}. Namely, for any $q>1$ there is a potential $V\in L^1(\R^2)$ such that $\lim_{\alpha\to\infty} \alpha^{-q} N(-\Delta+\alpha V)$ exists and is finite and positive. Therefore no bound like \eqref{eq:clr}, which would imply that $N(-\Delta+\alpha V)$ grows linearly with $\alpha$, can hold. In fact, the modified asymptotics hold for any long-range potential which behaves like $- |x|^{-2} (\ln |x|)^{-2} (\ln\ln |x|)^{-1/q}$ as $|x|\to\infty$. Moreover, if $q=1$ in these examples, then $\lim_{\alpha\to\infty} \alpha^{-1} N(-\Delta+\alpha V)$ exists, but is different from the right side of \eqref{eq:weyl}. We also mention that not only the slow decay at infinity can give rise to modified asymptotics but also strong (but integrable) local singularities. This can be understood via the conformal invariance of the problem. More concretely, the same modified asymptotics hold for any potential behaving like $- |x|^{-2} |\ln |x||^{-2} (\ln|\ln |x||)^{-1/q}$ as $|x|\to 0$. Note that the latter function belongs to $L^{\mathcal B}(\Omega)$ for an open set $\Omega$ containing $0$ if and only if $q<1$, so the modified asymptotics for $q\geq 1$ do not contradict our Corollary \ref{weyl}. We also note that these examples with local singularities show that even for the number of eigenvalues of $-\Delta+V$ below some fixed number $E<0$ there cannot be a bound of the form $C_E(1+ \int_{\R^2} V(x)_-\,dx)$.

This discussion raises the question of characterizing all $V\in L^1(\R^2)$ (or all $0\geq V\in L^2(\R^2)$) such that either $\limsup_{\alpha\to\infty} \alpha^{-1} N(-\Delta+\alpha V)<\infty$ or such that \eqref{eq:weyl} holds. This problem was solved in the radial case in \cite{LaSo}, but is still open in general. The eigenvalue bounds in \cite{So,KhMaWu,Sh,GrNa,LaSo2} can be understood as sufficient conditions for an asymptotically linear bound. Other, faster growing bounds can be found, for instance, in \cite{MoVa,St,We}.

\medskip

This paper is organized as follows. In Section \ref{sec:strategy} we explain the strategy of the proof of Theorem \ref{ch3:solomyak} and Corollary \ref{weyl} and reduce it to two main ingredients, namely a Sobolev-type inequality and a covering argument, which will be discussed in Sections \ref{sec:trudinger} and \ref{sec:covering}, respectively. We present all facts about Orlicz spaces which are relevant for us in Appendix \ref{sec:orlicz} and include a proof of the Besicovich theorem in Appendix \ref{sec:besicovich}.

\subsection*{Acknowledgements.} The authors are grateful to Timo Weidl for extensive discussions related to this material. They acknowledge partial support by the U.S. National Science Foundation through grant DMS-1363432 (R.L.F.) and by a grant of the Russian Federation Government under the supervision of a leading scientist at the Siberian Federal University, 14.Y26.31.0006, (A.L.).


\section{Strategy of the proof}\label{sec:strategy}

The proof of Theorem \ref{ch3:solomyak} is based on two ingredients, namely Sobolev-type inequalities and a covering argument. We present these ingredients in this section and then explain how to derive Theorem \ref{ch3:solomyak} from them.

\begin{proposition}\label{trudinger}
There is a constant $S_2>0$ such that for any open set $\Omega\subset\R^2$ of finite measure and any $u\in H^1_0(\Omega)$,
\begin{equation}
\label{eq:trudingerh10}
\int_\Omega \mathcal A( S_2 |u|^2/\|\nabla u\|^2 )\,dx \leq |\Omega| \,.
\end{equation}
Moreover, there is a constant $S_2'>0$ such that for any open set $\Omega\subset\R^2$, any open cube $Q\supset\Omega$ and any $u\in H^1(\Omega)$ which vanishes near $Q\cap\partial\Omega$ and satisfies $\int_\Omega u\,dx =0$,
\begin{equation}
\label{eq:pointrusolo}
\int_\Omega \mathcal A( S_2' |u|^2/\|\nabla u\|^2 )\,dx \leq |\Omega| \,.
\end{equation}
\end{proposition}

The crucial point for us is that the constants $S_2$ and $S_2'$ do not depend on $\Omega$. Note that we can take $\Omega=Q$ in the second part of the proposition and then \eqref{eq:pointrusolo} becomes
\begin{equation}
\label{eq:pointrusolocube}
\int_Q \mathcal A( S_2' |u|^2/\|\nabla u\|^2 )\,dx \leq |Q|
\end{equation}
for any $u\in H^1(Q)$ with $\int_Q u\,dx =0$. This inequality, however, is weaker than the second part of the proposition. Indeed, while it is true that functions $u\in H^1(\Omega)$ which vanishes near $Q\cap\partial\Omega$ and satisfy $\int_\Omega u\,dx =0$ can be extended by zero to functions in $H^1(Q)$ with mean value zero, applying \eqref{eq:pointrusolocube} to this extension gives \eqref{eq:pointrusolo} only with $|Q|$ on the right side and not with $|\Omega|$. This would not be sufficient for our purposes. The proof of Proposition \ref{trudinger} will be discussed in Section \ref{sec:trudinger}.

The second ingredient in the proof of Theorem \ref{ch3:solomyak} is the following covering result. By a \emph{cube} we always mean an open cube with edges parallel to the coordinate axes, and by a \emph{covering} of a set $K\subset\R^2$ by cubes $Q_1,\ldots, Q_M$ we mean that $K\subset\bigcup_j Q_j$. The \emph{multiplicity} of such a covering is $\sup_{x\in \R^2}\#\{j: \ x\in Q_j \}$.

\begin{proposition}\label{coveringorlicz}
Let $\Omega\subset\R^2$ be an open set of finite measure and let $0\leq W\in L^{\mathcal B}(\Omega)$ with compact support. Then for any $0< A \leq \|W\|_{\mathcal B,\Omega}$ there is a covering of $\Omega$ by open cubes $Q_1, \ldots, Q_M\subset\R^2$ of multiplicity at most $4$ such that 
\begin{equation}\label{eq:coveringorlicz}
\| W \|_{\mathcal B,Q_m\cap\Omega} = A \qquad\text{for all}\ 1\leq m\leq M \,.
\end{equation}
Moreover,
\begin{equation}
\label{eq:coveringorlicznumber}
M \leq 17 A^{-1} \|W\|_{\mathcal B,\Omega} \,.
\end{equation}
\end{proposition}

The proof of this proposition will be discussed in Section \ref{sec:covering}.

Having introduced our tools we are now in position to give the

\begin{proof}[Proof of Theorem \ref{ch3:solomyak}]
We first assume that $\supp V_-$ is bounded. We denote by $S_2$ and $S_2'$ the constants from Proposition \ref{trudinger}. If $\|V_-\|_{\mathcal B,\Omega}\leq S_2$, then we can use the first part of that proposition as well as the definition of $\|\cdot\|_{\mathcal B,\Omega}$ to bound for any $u\in H^1_0(\Omega)$
\begin{align*}
\int_\Omega \left( |\nabla u|^2 + V |u|^2 \right)dx & \geq \|\nabla u\|^2 \left( 1 - S_2^{-1} \int_\Omega V_- \frac{S_2 |u|^2}{\|\nabla u\|^2} \,dx \right) \\
& \geq \|\nabla u\|^2 \left( 1 - S_2^{-1} \|V_-\|_{\mathcal B,\Omega} \right) \geq 0 \,.
\end{align*}
Thus, $N(-\Delta_\Omega^D-V)=0$ and the theorem holds in this case.

We now assume that $\|V_-\|_{\mathcal B,\Omega}> S_2$. We apply Proposition \ref{coveringorlicz} with $W=V_-$ and $A= \min\{4^{-1} S_2',S_2\}$. We obtain a covering of $\Omega$ by open cubes $Q_1,\ldots,Q_M$ of multiplicity at most $4$ such that
$$
\| V_- \|_{\mathcal B,Q_m\cap\Omega} \leq 4^{-1} S_2'
\qquad\text{for all}\ m=1,\ldots, M
$$
and
$$
M \leq 17 \max\left\{ 4 (S_2')^{-1},S_2^{-1} \right\} \| V_- \|_{\mathcal B,\Omega} \,.
$$
Thus, for any $u\in H^1_0(\Omega)$ which satisfies the orthogonality conditions
$$
\int_{Q_m} u\,dx = 0
\qquad\text{for all}\ m=1,\ldots, M
$$
we can bound, using the second part of Proposition \ref{trudinger},
\begin{align*}
\int_\Omega \left( |\nabla u|^2 + V |u|^2 \right)dx & \geq \sum_{m=1}^M \int_{Q_m} \left( \frac14 |\nabla u|^2 - V_- |u|^2 \right)dx \\
& = \frac14 \sum_{m=1}^M \| \nabla u \|^2_{L^2(Q_m)} \left( 1 - \frac{4}{S_2'} \int_{Q_m\cap\Omega} V_- \frac{S_2' |u|^2}{\|\nabla u\|^2_{L^2(Q_m)}} \,dx \right) \\
& \geq \frac14 \sum_{m=1}^M \| \nabla u \|^2_{L^2(Q_m)} \left( 1 - \frac{4}{S_2'} \| V_- \|_{\mathcal B,Q_m\cap\Omega} \right)\geq 0 \,. 
\end{align*}
By the variational principle, this implies that $N(-\Delta_\Omega^D-V)\leq M$, and the upper bound on $M$ from the covering result proves the theorem in the case where $\supp V_-$ is bounded.

In the general case, we fix $\epsilon>0$ and, by a similar argument as in Lemma \ref{orliczbounded}, we choose $R>0$ such that $W:= \1_{\{|x|<R\}} V_-$ satisfies $\| W-V_- \|_{\mathcal B,\Omega}\leq \epsilon S_2$. Then for any $u\in H^1_0(\Omega)$ we have
\begin{align*}
\int_\Omega \left( |\nabla u|^2 + V |u|^2 \right)dx & \geq (1-\epsilon) \int_\Omega \left( |\nabla u|^2 - (1-\epsilon)^{-1}W |u|^2 \right) dx \\
& \qquad + \epsilon \int_\Omega \left( |\nabla u|^2 - \epsilon^{-1} (V_--W)|u|^2\right)dx \\
& \geq (1-\epsilon) \int_\Omega \left( |\nabla u|^2 - (1-\epsilon)^{-1}W |u|^2 \right) dx \\
& \qquad + \epsilon \int_\Omega |\nabla u|^2\,dx \left( 1- (\epsilon S_2)^{-1} \|V_- - W \|_{\mathcal B,\Omega} \right) \\
& \geq (1-\epsilon) \int_\Omega \left( |\nabla u|^2 - (1-\epsilon)^{-1}W |u|^2 \right) dx \,.
\end{align*}
By the variational principle, this implies $N(-\Delta_\Omega^D+V)\leq N(-\Delta_\Omega^D - (1-\epsilon)^{-1} W)$, and by the first part of the proof this can be bounded by $C (1-\epsilon)^{-1} \|W\|_{\mathcal B,\Omega}\leq C(1-\epsilon)^{-1} \|V_-\|_{\mathcal B,\Omega}$, where the last inequality follows easily from the definition of $\|\cdot\|_{\mathcal B,\Omega}$ norm (see also the proof of Lemma \ref{orliczmono}). The proof of Theorem \ref{ch3:solomyak} is complete.
\end{proof}

\begin{proof}[Proof of Corollary \ref{weyl}]
We use an approximation argument similarly to that at the end of the proof of Theorem \ref{ch3:solomyak}. For fixed $\epsilon>0$ and continuous and compactly supported $W$ we write
$$
-\Delta_\Omega^D+\alpha V = (1-\epsilon) \left(-\Delta_\Omega^D + \alpha (1-\epsilon)^{-1} W \right) + \epsilon \left(-\Delta_\Omega^D+\alpha \epsilon^{-1} (V-W) \right)
$$
and bound, using the variational principle and Theorem \ref{ch3:solomyak},
\begin{align*}
N(-\Delta_\Omega^D+\alpha V) 
& \leq N(-\Delta_\Omega^D+ \alpha (1-\epsilon)^{-1} W) + N(-\Delta_\Omega^D + \alpha \epsilon^{-1} (V-W)) \\
& \leq N(-\Delta_\Omega^D+ \alpha (1-\epsilon)^{-1} W) + C \alpha \epsilon^{-1} \| V-W\|_{\mathcal B,\Omega} \,.
\end{align*}
Using the Weyl asymptotics for continuous and compactly supported potentials, we obtain
\begin{equation}
\label{eq:weylproof1}
\limsup_{\alpha\to\infty} \alpha^{-1} N(-\Delta_\Omega^D +\alpha V) \leq (1-\epsilon)^{-1} \frac{1}{4\pi} \int_\Omega W(x)_- \,dx + C \epsilon^{-1} \| V-W\|_{\mathcal B,\Omega} \,.
\end{equation}
Similarly, we write
$$
-\Delta_\Omega^D+\alpha (1-\epsilon) W = (1-\epsilon) \left(-\Delta_\Omega^D + \alpha V \right) + \epsilon \left(-\Delta_\Omega^D+\alpha \epsilon^{-1} (1-\epsilon) (W-V) \right)
$$
and bound
\begin{align*}
N(-\Delta_\Omega^D+\alpha (1-\epsilon) W) 
& \leq N(-\Delta_\Omega^D+ \alpha V) + N(-\Delta_\Omega^D + \alpha \epsilon^{-1} (1-\epsilon)(W-V)) \\
& \leq N(-\Delta_\Omega^D+ \alpha V) + C \alpha \epsilon^{-1}(1-\epsilon) \| W-V\|_{\mathcal B,\Omega} \,.
\end{align*}
We obtain
\begin{equation}
\label{eq:weylproof2}
\liminf_{\alpha\to\infty} \alpha^{-1} N(-\Delta_\Omega^D+\alpha V) \geq (1-\epsilon) \frac1{4\pi} \int_\Omega W(x)_- \,dx - C \epsilon^{-1} (1-\epsilon) \|W-V\|_{\mathcal B,\Omega} \,.
\end{equation}
By an argument as in Lemma \ref{orliczbounded}, there is a sequence of continuous $W_n$ with compact support such that $\|W_n-V\|_{\mathcal B,\Omega}\to 0$. We also note that with $C'$ such that $\mathcal A(1/C')=1$, we have for any $g\in L^\mathcal B(\Omega)$,
$$
\| g \|_{L^1(\Omega)} = C' \int_\Omega \frac{\overline{\sgn g}}{C'} g \,dx \leq C' \| g\|_{\mathcal B,\Omega} \,,
$$
and therefore $\|W_n - V\|_{L^1(\Omega)}\leq C' \|W_n-V\|_{\mathcal B,\Omega}\to 0$. Replacing $W$ by $W_n$ in \eqref{eq:weylproof1} and \eqref{eq:weylproof2} and letting $n\to\infty$ we obtain
\begin{align*}
(1-\epsilon) \frac{1}{4\pi} \int_\Omega V(x)_-\,dx & \leq \liminf_{\alpha\to\infty} \alpha^{-1} N(-\Delta_\Omega^D+\alpha V) \\
& \leq \limsup_{\alpha\to\infty} \alpha^{-1} N(-\Delta_\Omega^D+\alpha V) \leq (1-\epsilon)^{-1} \frac{1}{4\pi} \int_\Omega V(x)_-\,dx \,.
\end{align*}
Since $\epsilon>0$ is arbitrary, the corollary follows.
\end{proof}


\section{Trudinger's inequality}\label{sec:trudinger}

Our goal in this section is to prove Proposition \ref{trudinger}. The first part of this proposition is well-known and goes back to the works \cite{Yu,Po,Tr}. Since we will need some intermediate result from this proof for the proof of the second part, we recall it here, following \cite{Og}.

\begin{lemma}\label{trudingerlem}
There is a constant $\alpha_0>0$ and a continuous function $[0,\alpha_0)\ni\alpha\to C_\alpha$ with $C_0=0$ such that for any open set $\Omega\subset\R^2$ of finite measure, any $u\in H^1_0(\Omega)$ and any $0\leq\alpha<\alpha_0$.
$$
\int_\Omega \mathcal A\left( \frac{\alpha |u|^2}{\|\nabla u\|^2} \right) dx \leq C_\alpha |\Omega| \,.
$$
\end{lemma}

\begin{proof}
We begin by showing that for $q>2$ and $u\in H^1(\R^2)$,
\begin{equation}
\label{eq:sobint}
\left( \int_{\R^2} |\nabla u|^2 \,dx \right)^{1-2/q} \left( \int_{\R^2} |u|^2\,dx \right)^{2/q} \geq \frac{2}{q} (4\pi)^{(q-2)/q} \left(\int_{\R^2} |u|^q \,dx \right)^{2/q} \,.
\end{equation}
The point here is the explicit expression of the constant on the right side and, in particular, its behavior as $q\to\infty$.

In order to prove this inequality we apply, with a parameter $\kappa>0$ to be determined, the Hausdorff--Young and the H\"older inequality to get
$$
\|u\|_q \leq (2\pi)^{-(q-2)/q} \|\hat u\|_{q'} \leq (2\pi)^{-(q-2)/q} \|(|\xi|^2+\kappa^2)^{1/2}\hat u\| \|(|\xi|^2 +\kappa^2)^{-1/2}\|_{2q/(q-2)} \,.
$$
Since
$$
\left\|(|\xi|^2 +\kappa^2)^{-1/2}\right\|_{2q/(q-2)}^{2q/(q-2)} = 2\pi \int_0^\infty \frac{k\,dk}{(k^2+\kappa^2)^{q/(q-2)}} = \frac{q-2}{2} \, \pi\, \kappa^{-4/(q-2)} \,,
$$
we obtain
$$
\|u\|_q^2 \leq (8\pi)^{-(q-2)/q} \left( q-2\right)^{(q-2)/q} \kappa^{-4/q} \left( \|\nabla u\|^2 + \kappa^2 \|u\|^2 \right).
$$
We optimize the right side by choosing $\kappa^2 = (2/(q-2)) \|\nabla u\|^2 /\|u\|^2$ and obtain \eqref{eq:sobint}.

If $u\in H^1_0(\Omega)$, we can use
$$
\int_\Omega |u|^2\,dx \leq |\Omega|^{1-2/q} \left( \int_\Omega |u|^q\,dx \right)^{2/q}
$$
and obtain from \eqref{eq:sobint}
$$
\left( \int_\Omega |\nabla u|^2\,dx \right)^{q/2} |\Omega| \geq \left( \frac{2}{q} \right)^{q^2/(2(q-2))} (4\pi)^{q/2} \int_{\Omega} |u|^q \,dx \,.
$$
Thus,
$$
\int_\Omega \mathcal A\left( \frac{\alpha |u|^2}{\|\nabla u\|^2} \right) dx = \sum_{n=2}^\infty \frac{1}{n!} \int_\Omega \left( \frac{\alpha |u|^2}{\|\nabla u\|^2} \right)^n dx 
\leq C_\alpha \, |\Omega| 
$$
with
$$
C_\alpha := \sum_{n=2}^\infty \frac{n^{n^2/(n-1)}}{n!} \left( \frac{\alpha}{4\pi} \right)^n \,.
$$
Using Stirling's asymptotics and the root test we see that $C_\alpha$ converges if $0\leq\alpha<4\pi/e$ and defines a continuous function with $C_0=0$.
\end{proof}

\begin{proof}[Proof of Proposition \ref{trudinger}]
Let $C_\alpha$ be the constant from Lemma \ref{trudingerlem}. For the proof of the first part, we simply choose $S_2>0$ such that $C_{S_2}\leq 1$.

For the proof of the second part, let $\tilde Q$ be the cube with the same center as $Q$ but with three times its side length and let $E: H^1(Q)\to H^1(\tilde Q)$ be the extension operator by repeated reflection. Thus, for all $u\in H^1(Q)$,
$$
\int_{\tilde Q} |Eu|^2 \,dx = 9 \int_Q |u|^2\,dx
\qquad\text{and}\qquad
\int_{\tilde Q} |\nabla Eu|^2\,dx = 9 \int_Q |\nabla u|^2\,dx \,.
$$
Let $\eta\in C_0^\infty(\tilde Q)$ be a real function with $\eta\equiv 1$ on $Q$. We choose this function to be of the form $\eta(x) = \eta_0((x-a)/L)$, where $a$ and $L$ are the center and the side length of $Q$, respectively, and where $\eta_0$ is a universal function. Then the operator $\tilde E$ defined by $\tilde E u = \eta E u$ maps $H^1(Q)$ into $H^1_0(\tilde Q)$ and satisfies
\begin{align*}
\int_{\tilde Q} |\nabla \tilde Eu|^2\,dx &
= \int_{\tilde Q} \left( \eta^2 |\nabla E u|^2 - \eta\Delta\eta |u|^2 \right) dx \notag \\
& \leq 9 \|\eta_0\|_\infty^2 \int_Q |\nabla u|^2 \,dx + 9 \|\eta_0\Delta\eta_0\|_\infty |Q|^{-1} \int_Q |u|^2 \,dx \,.
\end{align*}
This inequality is, in particular, valid for functions $u\in H^1(\Omega)$ which vanish near $Q\cap\partial\Omega$, because such functions can be extended by zero to functions in $H^1(Q)$. If, in addition, $\int_\Omega u\,dx =\int_Q u\,dx=0$, then we can bound the last term on the right side by the Poincar\'e inequality on $Q$ and we finally obtain
$$
\int_{\tilde Q} |\nabla \tilde Eu|^2\,dx \leq c \int_Q |\nabla u|^2\,dx
$$
with $c:=9 \left( \|\eta_0\|_\infty^2 + \pi^{-2} \|\eta_0\Delta\eta_0\|_\infty \right)$.

Moreover, let $\tilde\Omega\subset\tilde Q$ be the set obtained from $\Omega$ by repeated reflection on the boundaries of the cubes. Then $|\tilde\Omega|= 9 |\Omega|$. If $u\in H^1(\Omega)$ vanishes near $Q\cap\partial\Omega$, then $\tilde E u \in H^1_0(\tilde\Omega)$ and therefore by the inequality from Lemma \ref{trudingerlem},
$$
\int_{\Omega} \mathcal A\left( \frac{\alpha |u|^2}{\|\nabla u\|_{L^2(\Omega)}^2 }\right)dx
\leq \int_{\tilde\Omega} \mathcal A\left( \frac{c \alpha |\tilde E u|^2}{\|\nabla \tilde E u\|_{L^2(\Omega)} }\right)dx \leq C_{c\alpha} |\tilde \Omega| = 9C_{c\alpha} |\Omega| \,.
$$
Choosing $S_2'>0$ such that $9C_{cS_2'}\leq 1$, we obtain the claimed inequality.
\end{proof}


\section{The covering lemma}\label{sec:covering}

Our goal in the section is to prove Proposition \ref{coveringorlicz}. In a first step we will see that around each point we can center a cube for which the norm of a given function has a prescribed Orlicz norm. This requires some basic facts about Orlicz spaces, which we recall in Appendix \ref{sec:orlicz}. In a second step we apply the Besicovich theorem to obtain a suitable finite collection of cubes. The relevant version of the Besicovich theorem will be recalled in Appendix \ref{sec:besicovich}.

\begin{lemma}\label{orlicznormcont}
Let $\Omega\subset\R^2$ be an open set of finite measure, let $0\leq W\in L^{\mathcal B}(\Omega)$ and let $0< A<\| W \|_{\mathcal B,\Omega}$. Then for any $x\in\overline\Omega$ there is an open cube $Q_x$ centered at $x$ with
$$
\| W \|_{\mathcal B,Q_x\cap\Omega}=A \,.
$$
If $Q_x$ is chosen maximal with this property, then $\overline\Omega\ni x\mapsto |Q_x|$ is upper semi-continuous.
\end{lemma}

\begin{proof}
First, we fix $x\in\overline\Omega$ and consider the function $j(l) := \| V \|_{\mathcal B,(x+l Q)\cap\Omega}$, where $Q:=(-1/2,1/2)^d$, so $x+lQ$ is the open cube centered at $x$ with side length $l$. By a simple property of the norm (see \eqref{eq:orliczsubaddcor}), $j$ is a non-decreasing function of $l$.

We claim that $j$ is continuous on $[0,\infty)$ with $j(0)=0$. To prove this, we let $(l_n)\subset(0,\infty)$ with $l_n\to l\in [0,\infty)$. Setting $E_n:=(x+l_n Q)\cap\Omega$ and $E:=(x+lQ)\cap\Omega$, we clearly have $|E_n\Delta E|\to 0$ as $n\to\infty$ and therefore, according to Lemma \ref{orliczcont},
$$
j(l_n) = \|V\|_{\mathcal B,E_n} \to \|V\|_{\mathcal B,E}=j(l) \,,
$$
proving the claimed continuity and the fact that $j(0)=0$.

Also, it is easy to see that $\lim_{l\to\infty} j(l) = \|V\|_{\mathcal B,\Omega}$. Thus, for any $0<A<\|V\|_{\mathcal B,\Omega}$ there is an $l$ such that $j(l) = A$. We denote $l_x:=\max\{ l:\ j(l)=A\}$, making the dependence on $x$ explicit.

We now prove the upper semi-continuity statement. This will follow if we can show that for $(x_n)\subset\overline\Omega$ and $(l_n)\subset (0,\infty)$ with $x_n\to x\in\overline\Omega$ and $l_n\to l\in (0,\infty)$ one has
$$
\| V \|_{\mathcal B,(x_n+l_n Q)\cap\Omega} \to \| V \|_{\mathcal B,(x+l Q)\cap\Omega} \,.
$$
Indeed, if we apply this statement to $l_n=l_{x_n}$, then we obtain that $\| V \|_{\mathcal B,(x+lQ)\cap\Omega} = A$, and by maximality of $l_x$ we conclude that $l_x \geq l$, which proves upper semi-continuity.

To prove the statement above, we apply again Lemma \ref{orliczcont}, this time with $E_n := (x_n+l_n Q)\cap\Omega$ and $E:=(x+lQ)\cap\Omega$. Again one easily checks that $|E_n\Delta E|\to 0$ as $n\to\infty$, so the assumption of this lemma is satisfied. This finishes the proof of the lemma.
\end{proof}

We are now in position to give the

\begin{proof}[Proof of Proposition \ref{coveringorlicz}]
We may assume that $0<A<\|W\|_{\mathcal B,\Omega}$. Then Lemma \ref{orlicznormcont} yields for any point $x\in\overline\Omega$ an open cube $Q_x$ centered at $x$ with $\| W \|_{\mathcal B,Q_x\cap\Omega}=A$. Moreover, the side length $|Q_x|^{1/2}$ depends in an upper semi-continuous way on $x$. Thus, the Besicovitch lemma (Proposition \ref{ch3:besicovitch}) yields a countable covering of multiplicity $4$ and with the property that the cubes can be divided into families $\Xi^k$, $k=1,\ldots,17$, each of which consists of disjoint cubes.

It remains to show that the covering is finite and with the claimed upper bound on the number $M$ of cubes. For any $k=1,\ldots,17$, by the superadditivity property of the Orlicz norm (Lemma \ref{orliczsubadd}),
$$
(\# \Xi^k) A = \sum_{Q\in\Xi^k} \|W\|_{\mathcal B,Q\cap\Omega} \leq \|W\|_{\mathcal B,\Omega} \,.
$$
Summing over $k$ we obtain
$$
M A \leq 17\, \|W\|_{\mathcal B,\Omega} \,,
$$
which is the claimed bound.
\end{proof}



\appendix

\section{Orlicz spaces}\label{sec:orlicz}

In order to make this paper self-contained in this appendix we provide proofs of the results from Orlicz space theory which we need. For a deeper treatment we refer, for instance, to \cite{KrRu}. 

Throughout this section we consider a convex function $\mathcal A$ on $[0,\infty)$ satisfying $\mathcal A(t)=0$ if and only if $t=0$, as well as
$$
\lim_{t\to\infty} \frac{\mathcal A(t)}{t}=\infty
\qquad\text{and}\qquad
\lim_{t\to 0} \frac{\mathcal A(t)}{t} = 0 \,.
$$
(Such functions are called \emph{Young functions}.) The example relevant for the rest of this paper is the function
$$
\mathcal A(t) = e^{|t|} -1-|t| \,,
$$
but our arguments are valid for general $\mathcal A$.

Let $\mathcal B$ be the Legendre transform of $\mathcal A$, that is,
$$
\mathcal B(s) = \sup_{t\geq 0} \left(st - \mathcal A(t)\right)
\qquad\text{for}\ s\geq 0 \,.
$$
It can be shown that $\mathcal B$ is again a Young function. In the example above, we have
$$
\mathcal B(s) = (1+|s|)\ln(1+|s|) - |s| \,. 
$$
For a finite measure space $(X,dx)$ we denote by $L^\mathcal B(X)$ the set of measurable functions $g:X\to \C$ for which
$$
\| g\|_{\mathcal B,X} = \sup\left\{ \left| \int_X fg\,dx \right| :\ f:X\to\C \ \text{measurable such that}\ \int_X \mathcal A(|f|) \,dx \leq |X| \right\}
$$
is finite (identifying almost everywhere equal functions). Clearly, $\|\cdot\|_{\mathcal B,X}$ defines a norm. We first show that this norm is superadditive.

\begin{lemma}\label{orliczsubadd}
Let $g\in L^\mathcal B(X)$ and let $E_1,E_2,\ldots$ be pairwise disjoint measurable subsets of $X$. Then
$$
\sum_j \|g\|_{\mathcal B,E_j} \leq \|g\|_{\mathcal B,X} \,.
$$
\end{lemma}

Note that this implies, in particular, that
\begin{equation}
\label{eq:orliczsubaddcor}
\|g\|_{\mathcal B,E} \leq \|g\|_{\mathcal B,X}
\qquad\text{if}\ E\subset X \,.
\end{equation}

\begin{proof}
Consider any sequence of measurable functions $f_1,f_2,\ldots$ on $E_1,E_2,\ldots$ with
$$
\int_{E_j} \mathcal A(|f_j|) \,dx \leq |E_j|
\qquad\text{for all}\ j \,.
$$
We define functions $s_j$ on $E_j$ with $|s_j|=1$ and $s_j f_j g =|f_j g|$ pointwise on $E_j$. We define a function $\tilde f$ on $X$ by $\tilde f|_{E_j}:=s_j f_j$ for each $j$ and $\tilde f:=0$ on $X\setminus\bigcup_j E_j$. Then
$$
\int_X \mathcal A(|\tilde f|) \,dx = \sum_j \int_{E_j} \mathcal A(|s_j f_j|) \,dx = \sum_j \int_{E_j} \mathcal A(|f_j|) \,dx \leq \sum_j |E_j| \leq |E| \,.
$$
Thus,
$$
\sum_j \left| \int_{E_j} f_j g \,dx \right| = \sum_j \int_{E_j} \tilde f_j g \,dx = \int_X \tilde f g \,dx \leq \|g\|_{\mathcal B,X} \,.
$$
Taking the supremum over all functions $f_j$ with the specified properties we arrive at the inequality in the lemma.
\end{proof}

\begin{lemma}\label{orliczmono}
Let $g\in L^\mathcal B(X)$, let $E,F\subset X$ be measurable with $|E|\leq |F|$ and assume that $g$ vanishes on $X\setminus(E\cap F)$. Then
$$
\|g\|_{\mathcal B,E} \leq \|g\|_{\mathcal B,F} \leq \frac{|F|}{|E|} \|g\|_{\mathcal B,E} \,.
$$
\end{lemma}

\begin{proof}
We first observe that, since $f$ vanishes off $E\cap F$ and since $\mathcal A$ is monotone, we have
$$
\|g\|_{\mathcal B,E} = \sup\left\{ \left| \int_{E\cap F} fg\,dx \right| : f:E\cap F\to\C \ \text{measurable with} \int_{E\cap F} \mathcal A(|f|) \,dx \leq |E| \right\}
$$
and
$$
\|g\|_{\mathcal B,F} = \sup\left\{ \left| \int_{E\cap F} fg\,dx \right| : f:E\cap F\to\C \ \text{measurable with} \int_{E\cap F} \mathcal A(|f|) \,dx \leq |F| \right\}.
$$
Since $|E|\leq |F|$, this immediately implies $\|g\|_{\mathcal B,E}\leq \|g\|_{\mathcal B,F}$.

To prove the converse inequality, let $f$ be a measurable function on $E\cap F$ with $\int_{E\cap F} \mathcal A(|f|) \,dx \leq |F|$. Since $\mathcal A$ is convex with $\mathcal A(0)=0$ we have $\mathcal A(\theta t)\leq \theta\mathcal A(t)$ for any $0\leq\theta\leq 1$ and any $t\geq0$. Thus, $\tilde f:=(|E|/|F|) f$ satisfies $\int_{E\cap F}\mathcal A(|\tilde f|)\,dx \leq |E|$, and therefore
$$
\|g\|_{\mathcal B,E} \geq \left| \int_{E\cap F} \tilde f g\,dx \right| = \frac{|E|}{|F|} \left| \int_{E\cap F} f g\,dx \right| \,.
$$
Taking the supremum over all $f$ as before, we deduce $\|g\|_{\mathcal B,E}\geq (|E|/|F|) \|g\|_{\mathcal B,F}$, as claimed.
\end{proof}

\begin{lemma}\label{luxemburg}
For any $g\in L^\mathcal B(X)$,
$$
\int_X \mathcal B\left( \frac{|X|\, |g|}{\|g\|_{\mathcal B,X}} \right)dx \leq |X| \,.
$$
\end{lemma}

\begin{proof}
As a preliminary remark we note that
\begin{equation}
\label{eq:legendre0}
s\mathcal B'(s) = \mathcal A(\mathcal B'(s)) + \mathcal B(s)
\qquad\text{for all}\ s\geq 0 \,,
\end{equation}
where here and in what follows we denote by $\mathcal B'$ the right sided derivative of $\mathcal B$ (which exists by convexity). In fact, by convexity $\mathcal B(\sigma) \geq \mathcal B(s) + \mathcal B'(s)(\sigma-s)$ for all $\sigma$ and therefore $s \mathcal B'(s) - \mathcal B(s) = \sup_\sigma \left( \sigma \mathcal B'(s) - \mathcal B(\sigma) \right)$. Since $\mathcal A$ is the Legendre transform of $\mathcal B$, we obtain \eqref{eq:legendre0}.

We now turn to the proof of the lemma. Clearly, we may assume that $\|g\|_{\mathcal B,X}=|X|$. Let $f := \overline{\sgn g} \mathcal B'(|g|)$. We shall show momentarily that 
\begin{equation}
\label{eq:sizea}
\int_X \mathcal A(|f|)\,dx\leq |X| \,.
\end{equation}
Because of this inequality and the definition of $\|g\|_{\mathcal B}$, we have
$$
|X| = \| g\|_{\mathcal B,X} \geq \left| \int_X f g\,dx \right|
= \int_X |f| |g|\,dx \,
$$
On the other hand, because of \eqref{eq:legendre0} we have
\begin{equation}
\label{eq:legendre}
|f| |g| = \mathcal A(|f|) + \mathcal B(|g|) \,,
\end{equation}
and therefore
$$
\int_X |f| |g|\,dx = \int_X \mathcal A(|f|)\,dx + \int_X \mathcal B(|g|)\,dx \geq \int_X \mathcal B(|g|)\,dx \,,
$$
which yields the inequality in the lemma.

It remains to prove \eqref{eq:sizea}. For $M>0$ let $f_M:=f \1_{\{|f|\leq M\}}$ and note that \eqref{eq:legendre} yields
$$
|f_M| |g| =\mathcal A(|f_M|) + \1_{\{|f|\leq M\}} \mathcal B(|g|) \geq \mathcal A(|f_M|) \,.
$$
We choose $M$ large enough such that $\1_{\{|f|\leq M\}} \mathcal B(|g|)$ does not vanish almost everywhere and obtain
\begin{equation}
\label{eq:luxemburgproof}
\int_X |f_M| |g| \,dx > \int_X \mathcal A(|f_M|)\,dx \,.
\end{equation}

We now show that $\alpha := \int_X \mathcal A(|f_M|)\,dx \leq |X|$, which implies \eqref{eq:sizea} by monotone convergence. We argue by contradiction and assume that $\alpha>|X|$. Note that $\alpha<\infty$ since $f_M$ is bounded. As in the previous proof, by convexity, we have $\mathcal A(|X|\,|f_M|/\alpha) \leq (|X|/\alpha) \mathcal A(|f_M|)$ and therefore $\int_X \mathcal A(|X|\,|f_M|/\alpha)\,dx \leq |X|$. By definition of $\|g\|_{\mathcal B,X}$,
$$
\int_X |f_M| |g| \,dx = \frac{\alpha}{|X|} \int_X \frac{|X|\, f_M}{\alpha} g \,dx \leq \frac{\alpha}{|X|} \|g\|_{\mathcal B} = \alpha \,.
$$
This contradicts \eqref{eq:luxemburgproof}, and therefore we obtain $\alpha\leq |X|$.
\end{proof}

\begin{lemma}\label{orliczbounded}
Assume that $\mathcal B$ satisfies
\begin{equation}
\label{eq:orliczdelta2}
\limsup_{t\to\infty} \frac{\mathcal B(2t)}{\mathcal B(t)}<\infty \,.
\end{equation}
Then $L^\infty(X)$ is dense in $L^\mathcal B(X)$.
\end{lemma}

In the theory of Orlicz spaces, \eqref{eq:orliczdelta2} is called $\Delta_2$ \emph{condition}. Note that the function $\mathcal B(t) = (1+|t|)\ln(1+|t|)-|t|$ satisfies this condition (while $\mathcal A(t)= e^{|t|}-1-|t|$ does not).

\begin{proof}
Let $g\in L^\mathcal B(X)$. We show that  $g_M := g \1_{\{|g|\leq M\}} \to g$ in $L^\mathcal B(X)$ as $M\to\infty$.

Let $\lambda:=\|g\|_{\mathcal B,X}/|X|$, so that $\int_X \mathcal B(|g|/\lambda)\,dx<\infty$ by Lemma \ref{luxemburg}. Moreover, $\mathcal B(|g_M-g|/\lambda)\leq \mathcal B(|g|/\lambda)$ and therefore, by dominated convergence, $\int_X \mathcal B(|g_M-g|/\lambda)\,dx \to 0$ as $M\to\infty$. It is easy to see that assumption \eqref{eq:orliczdelta2} implies that for any $\epsilon>0$ and $k\in\N$ there is a $C_{k,\epsilon}<\infty$ such that
$$
\mathcal B(2^k t) \leq C_{k,\epsilon} \mathcal B(t) + \epsilon 
\qquad\text{for all}\ t\geq 0 \,. 
$$
Let $f$ be a measurable function with $\int_X \mathcal A(|f|)\,dx\leq |X|$. Then by the definition of $\mathcal B$ as Legendre transform,
$$
|f| 2^k|g_M-g|/\lambda \leq \mathcal A(|f|) + \mathcal B(2^k |g_M-g|/\lambda) \leq \mathcal A(|f|) + C_{k,\epsilon} \mathcal B(|g_M-g|/\lambda) + \epsilon \,.
$$
Thus,
$$
\left| \int_X f 2^k (g_M-g)/\lambda \,dx \right| \leq (1+\epsilon) |X| + C_{k,\epsilon} \int_X \mathcal B(|g_M-g|/\lambda)\,dx
$$
and, taking the supremum over $f$,
$$
(2^k/\lambda) \|g_M-g\|_{\mathcal B,X} = \| 2^k (g_M-g)/\lambda \|_{\mathcal B,X} \leq (1+\epsilon) |X| + C_{k,\epsilon} \int_X \mathcal B(|g_M-g|/\lambda)\,dx \,.
$$
Letting $M\to\infty$ gives
$$
\limsup_{M\to\infty} \| g_M-g \|_{\mathcal B,X} \leq \lambda 2^{-k} (1+\epsilon) |X| \,,
$$
and letting $k\to\infty$ gives $g_M\to g$ in $L^\mathcal B(X)$.
\end{proof}

\begin{lemma}\label{orliczcont}
Assume that $\mathcal B$ satisfies \eqref{eq:orliczdelta2}. Let $g\in L^\mathcal B(X)$ and let $E_1,E_2,\ldots$ and $E$ be measurable subsets of $X$ with $|E_n\Delta E|\to 0$ as $n\to\infty$. Then, as $n\to\infty$,
$$
\|g\|_{\mathcal B,E_n} \to \|g\|_{\mathcal B,E} \,.
$$
\end{lemma}

\begin{proof}
We first claim that we may assume that $g$ is bounded. Indeed, from Lemma \ref{orliczbounded} we know for any $\epsilon>0$ there is an $g_\epsilon\in L^\infty(X)$ such that $\|g_\epsilon-g\|_{\mathcal B,X}\leq \epsilon$. Thus, by the triangle inequality and by \eqref{eq:orliczsubaddcor} we have for any measurable $F\subset X$, $| \|g_\epsilon\|_{\mathcal B,F} - \|g\|_{\mathcal B,F}| \leq \|g_\epsilon-g\|_{\mathcal B,F} \leq \|g_\epsilon-g\|_{\mathcal B,X}\leq \epsilon$. Applying this with $F=E_n$ and with $F=E$, we see that it is enough to prove the lemma for $g\in L^\infty(X)$, as claimed.

Let us define $g_n:=g\1_{E_n\cap E}$. Then
\begin{align*}
\left| \|g\|_{\mathcal B,E_n} - \|g\|_{\mathcal B,E} \right| 
\leq \left| \|g\|_{\mathcal B,E_n} - \|g_n\|_{\mathcal B,E_n} \right|
& + \left| \|g_n\|_{\mathcal B,E_n} - \|g_n\|_{\mathcal B,E} \right| \\
& + \left| \|g_n\|_{\mathcal B,E} - \|g\|_{\mathcal B,E} \right| \,.
\end{align*}
We have, using \eqref{eq:orliczsubaddcor},
\begin{align*}
\left| \|g_n\|_{\mathcal B,E_n} - \|g\|_{\mathcal B,E_n} \right| & \leq \| g_n-g\|_{\mathcal B,E_n} = \| g \chi_{E_n\setminus E} \|_{\mathcal B,E_n} \leq \|g\|_\infty \| \chi_{E_n\setminus E} \|_{\mathcal B,E_n} \\
& \leq \|g\|_\infty \| \chi_{E_n\setminus E} \|_{\mathcal B,X} \,.
\end{align*}
and similarly
\begin{align*}
\left| \|g_n\|_{\mathcal B,E} - \|g\|_{\mathcal B,E} \right| & \leq \| g_n-g\|_{\mathcal B,E} = \| g\chi_{E\setminus E_n} \|_{\mathcal B,E} \leq \|g\|_\infty \| \chi_{E\setminus E_n} \|_{\mathcal B,E} \\
& \leq \|g\|_\infty \| \chi_{E\setminus E_n} \|_{\mathcal B,X} \,.
\end{align*}
It is a simple exercise (using Jensen's inequality) to compute that for any measurable $F\subset X$,
$$
\| \chi_F \|_{\mathcal B,X} = |F|\ \mathcal A^{-1}(|X|/|F|) \,,
$$
where $\mathcal A^{-1}$ is the inverse function of $\mathcal A$; see \cite[Subsection II.9.3]{KrRu}. Since $\mathcal A(t)/t\to \infty$ as $t\to\infty$, we deduce that $\| \chi_F \|_{\mathcal B,X}\to 0$ as $|F|\to 0$. Thus, the assumption $|E_n\Delta E|\to 0$ implies that
$$
\left| \|g\|_{\mathcal B,E_n} - \|g_n\|_{\mathcal B,E_n} \right| + \left| \|g_n\|_{\mathcal B,E} - \|g\|_{\mathcal B,E} \right| \to 0
$$
as $n\to\infty$.

According to \eqref{eq:orliczsubaddcor} and Lemma \ref{orliczmono} we have
$$
\left| \|g_n\|_{\mathcal B,E_n} - \|g_n\|_{\mathcal B,E} \right| \leq \left( \frac{\max\{|E_n|,|E|\}}{\min\{|E_n|,|E|\}} - 1\right) \|g\|_{\mathcal B,X} \,.
$$
Again, since $|E_n\Delta E|\to 0$, we deduce that
$$
\|g_n\|_{\mathcal B,E_n} - \|g_n\|_{\mathcal B,E} \to 0
$$
as $n\to\infty$. This completes the proof.
\end{proof}


\section{The Besicovich lemma}\label{sec:besicovich}

In this section we state and prove a version of Besicovich's covering lemma. We follow the exposition in \cite{dG}, but since we get a better constant under an additional semi-continuity assumption (which is satisfied in our application), we include the details. We prove the result in general dimension $d$.

We recall that by a \emph{cube} we always mean an open cube with edges parallel to the coordinate axes, and by a \emph{covering} of a set $K\subset\R^d$ by cubes $Q_1,\ldots, Q_M$ we mean that $K\subset\bigcup_j Q_j$. The \emph{multiplicity} of such a covering is $\sup_{x\in \R^d}\#\{j: \ x\in Q_j \}$.

We denote $Q:=(-1/2,1/2)^d$, so that $a+ lQ$ is the cube centered at $a\in\R^d$ with side length $l>0$.

\begin{proposition}\label{ch3:besicovitch}
Let $d\geq 1$, let $K\subset\R^d$ be a compact set and let $l$ be a positive, upper semi-continuous function on $K$. Then there is a (finite or infinite) sequence $(x_j)\subset K$ such that the cubes $Q_j=x_j+l(x_j) Q$, $j=1,2,3,\ldots,$ are a covering of $K$ with multiplicity at most $2^d$. Moreover, the sequence $(Q_j)$ can be divided into $4^d+1$ sequences $\Xi^k=(Q^{k}_j)$ such that for any $1\leq k\leq 4^d+1$, the cubes in $\Xi^k$ are disjoint (that is, $Q_{j_1}^k \cap Q_{j_2}^k =\emptyset$ if $j_1\neq j_2$).
\end{proposition}

In the proof we use the notation $|x|_\infty =\max\{|x^{(j)}|: 1\leq j\leq d\}$ for $x=(x^{(1)},\ldots,x^{(d)})\in\R^d$.

\begin{proof}
Since the function $l$ is upper semi-continuous on the compact set $K$, it attains its maximum at some point $x_1\in K$. Now assume that for some $m\in\N$, the points $x_1,\ldots,x_m$ have already been chosen. If $K\setminus \bigcup_{j=1}^m Q_j =\emptyset$, then the selection process is finished. Otherwise, we take $x_{m+1}\in K\setminus \bigcup_{j=1}^m Q_j$ such that the maximum of $l$ over the compact $K\setminus \bigcup_{j=1}^m Q_j$ is attained at $x_{m+1}$. This procedure leads to a finite or infinite sequence of points $x_j$. Let us show that they have all the required properties.

We claim that
\begin{equation}
\label{eq:besiproof1}
\left( x_i + \frac{l(x_i)}2 Q\right) \cap \left( x_j + \frac{l(x_j)}2 Q \right) =\emptyset
\qquad\text{if}\ i\neq j \,.
\end{equation}
Indeed, to show this we may assume that $i<j$. Then, by construction $x_j \not\in Q_i$ and therefore $|x_i-x_j|_\infty \geq l(x_i)/2$. By construction we have $l(x_i)\geq l(x_j)$ and therefore $|x_i-x_j|_\infty \geq (l(x_i)+l(x_j))/4$. This implies \eqref{eq:besiproof1}.

We also claim that
\begin{equation}
\label{eq:besiproof3}
x_i\not\in Q_j
\qquad\text{if}\ i\neq j \,.
\end{equation}
Indeed, this is clear from the construction if $i>j$. On the other hand, if $i<j$, then, again from the construction, we have $l(x_i)\geq l(x_j)$ and $x_j\not\in Q_i$, which implies that $|x_j-x_i|_\infty\geq l(x_i)/2\geq l(x_j)/2$. Thus, \eqref{eq:besiproof3} also holds in this case.

Using the compactness of $K$, one easily deduces from \eqref{eq:besiproof1} that, if the sequence $(Q_j)$ is infinite, then
\begin{equation}
\label{eq:besiproof2}
l(x_j)\to 0
\qquad\text{as}\ j\to \infty \,.
\end{equation}

We now prove that $(Q_j)$ covers $K$. This is clear from the construction if the sequence $(Q_j)$ is finite. So we may assume that it is infinite. We argue by contradiction and assume that there is an $x\in K \setminus \bigcup_j Q_j$. Then, because of \eqref{eq:besiproof2}, there is a $j$ such that $l(x_j)<l(x)$. This, however, contradicts the construction of $x_j$.

Next, we show that the multiplicity of the covering is at most $2^d$, that is, any point $x\in K$ belongs to at most $2^d$ of the cubes $Q_j$. To do this, we divide $\R^d$ into $2^d$ hyper-quadrants with boundaries passing through $x$ and parallel to the $d$ coordinate hyper-planes. It suffices to show that in each closed hyper-quadrant there is at most one point $x_j$ such that $x\in Q_j$. We argue by contradiction and assume that there are two distinct points $x_i$ and $x_j$ in the same closed hyper-quadrant with $x\in Q_i\cap Q_j$. We may assume that $|x_j-x|_\infty\geq |x_i-x|_\infty$. Since $x\in Q_j$, the set of all points $y$ in the same hyper-quadrant as $x_j$ satisfying $|y-x|_\infty \leq |x_j-x|_\infty$ is contained in $Q_j$. In particular, we have $x_i\in Q_j$. This contradicts \eqref{eq:besiproof3}.

Finally, we have to rearrange the sequence into $4^d+1$ disjoint sequences. We first claim that for any $j$ there are at most $4^d$ cubes among the cubes $Q_1,\ldots,Q_{j-1}$ which have non-empty intersection with $Q_j$. To see this, note that if $k<j$ and if $Q_k\cap Q_j \neq\emptyset$, then $Q_k$ contains at least one of the $2^d$ vertices of $Q_j$. (This follows from the fact that $l(x_k)\geq l(x_j)$ for $k<j$.) However, by the bound on the multiplicity, any fixed vertex of $Q_j$ is contained in at most $2^d$ cubes. Thus, there are at most $2^d \times 2^d$ cubes $Q_k$ with $k<j$ which have non-empty intersection with $Q_j$.

We now use this fact to decompose our sequence. We are going to define inductively $r=4^d+1$ sequences $\Xi^1,\ldots, \Xi^r$ of cubes. To start, we set $Q_j\in\Xi^j$ for $j=1,\ldots, r$. Now let $j\geq r+1$ and assume that the families $\Xi^1,\ldots,\Xi^r$ contain all the cubes $Q_1,\ldots, Q_{j-1}$ and that each $\Xi^k$ consists of disjoint cubes. By the above fact, $Q_{j}$ can intersect at most $r-1$ cubes among the cubes $Q_1,\ldots,Q_{j-1}$. Since there are $r$ families of cubes in total, there must be a $k\in\{1,\ldots,r\}$ such that $Q_{j}$ does not intersect any of the cubes in $\Xi^k$. We put $Q_{j}\in\Xi^k$. This defines inductively the claimed partitioning of the $Q_j$. The proof of the proposition is complete.
\end{proof}


\bibliographystyle{amsalpha}

\end{document}